\documentclass[a4paper,12pt,twoside]{amsart}
\usepackage[english]{babel}
\usepackage{amssymb, amsmath, eucal, enumerate}
\usepackage[all]{xy}

\usepackage[left=2.5cm,right=2.5cm,top=2cm,bottom=2.5cm]{geometry}

\usepackage[svgnames,hyperref]{xcolor}
\newcommand{\mycolor}{Navy}
\usepackage[pdfauthor={Son H. Do and  Dieu Q. Nguyen}, colorlinks, linktocpage, citecolor = \mycolor, linkcolor = \mycolor, urlcolor = \mycolor]{hyperref}
\newtheorem{The}{Theorem}[section]
\newtheorem{Lem}[The]{Lemma}
\newtheorem{Prop}[The]{Proposition}
\newtheorem{Cor}[The]{Corollary}
\newtheorem{Rem}[The]{Remark}

\newtheorem{Def}[The]{Definition}

\newtheorem*{ackn}{Acknowledgements}
\newcommand{\B}{\mathbb{B}}

\newcommand{\C}{\mathbb{C}}
\newcommand{\R}{\mathbb{R}}

\newcommand{\Z}{\mathbb{Z}}

\newcommand{\Jj}{\mathcal{J}}

\begin{document}
  \title[Fully nonlinear elliptic equations]{On the viscosity approach to a class of fully nonlinear
  elliptic equations} 
\setcounter{tocdepth}{1}
\author{Hoang-Son Do} 
\address{Institute of Mathematics \\ Vietnam Academy of Science and Technology \\18
Hoang Quoc Viet \\Hanoi \\Vietnam}
\email{hoangson.do.vn@gmail.com, dhson@math.ac.vn}
\author{Quang Dieu Nguyen}
\address{Department of Mathematics, Hanoi National University of Education, 136-Xuan Thuy, Cau Giay, Hanoi, VietNam}
\email{dieu\_vn@yahoo.com}

\date{\today\\The first author was supported by Vietnam Academy of Science and Technology under grant number CT0000.07/21-22. The second author was supported by the Vietnam National Foundation for Science and Technology Development (NAFOSTED) under grant number 101.02-2019.304}
\maketitle
\begin{abstract}
In this paper, we study some properties of viscosity sub/super-solutions of a class of fully nonlinear elliptic equations relative to the eigenvalues of the complex Hessian. We show that every viscosity subsolution is
approximated by a decreasing sequence of smooth subsolutions. When the equations satisfy some conditions on the
limit at infinity, we verify that the comparison principle holds, and as a sequence, we obtain a result about
the existence of solution of the Dirichlet problem. Using the comparison principle, we show that, under suitable conditions, a 
Perron-Bremermann envelope can be approximated by a decreasing sequence of viscosity solutions.
\end{abstract}
\tableofcontents
\section{Introduction}
Let $\Gamma\varsubsetneq\R^n$ be an open, convex, symmetric cone with vertex at the origin such that $\Gamma_n\subseteq\Gamma\subseteq\Gamma_1$,
where
	\begin{center}
		$\Gamma_k=\{x\in\R^n: \sigma_1(x)>0,...,\sigma_k(x)>0 \},$
	\end{center}
	for every $1\leq k\leq n$. Here $\sigma_k(x)$ is the $k$-th elementary symmetric sum of the coefficients of $x$, i.e.,
	\begin{center}
		$\sigma_k(x)=\sum\limits_{1\leq j_1<...<j_k\leq n}x_{j_1}...x_{j_k}.$
	\end{center}
 Let $f:\overline{\Gamma}\rightarrow (0, \infty)$
 be a continuous function such that
  \begin{itemize}
 	\item $f$ is symmetric, strictly increasing in each variable and concave.
 	\item $f|_{\Gamma}>0$, $f|_{\partial\Gamma}=0$.
 \end{itemize}
  We define $F:\mathcal{H}^n\rightarrow [-\infty, \infty)$ by
  \begin{equation}
  F(H)=\begin{cases}
  f(\lambda(H))\quad\mbox{if}\quad H\in \overline{M(\Gamma, n)},\\
  -\infty\quad\mbox{if}\quad H\in \mathcal{H}^n\setminus\overline{M(\Gamma, n)},
  \end{cases}
  \end{equation}
   where $\mathcal{H}^n$ is the set of all $n\times n$ Hermitian matrices and
  $M(\Gamma, n)$ is the subset of $\mathcal{H}^n$ containing
matrices $H$ with the eigenvalues $\lambda (H)=(\lambda_1,...,\lambda_n)\in\Gamma$. The conditions on $f$ imply that $F$ is concave
on $\overline{M(\Gamma, n)}$ (see \cite{CNS}) and
\begin{equation}\label{F is increasing}
F(M+N)>F(M),
\end{equation}
for every $M\in M(\Gamma, n)$ and for each positive semidefinite  matrix $N\neq 0$.\\

  We consider the fully nonlinear elliptic equation:
  \begin{equation}\label{main free boundary}
  F(Hu)=\psi(z, u),
  \end{equation}
 in a bounded domain $\Omega\subset\C^n$, where $Hu$ is the complex Hessian of $u$ and
  $\psi\geq 0$ is a continuous function in $\overline{\Omega}\times\R$ which is non-decreasing in the last variable.
 
 In the smooth setting, 
 the existence and uniqueness of the classical solution of the Dirichlet problem of \eqref{main free boundary}
 has been studied in \cite{Li} (see also \cite{Sze}  for corresponding problems on compact Hermitian manifolds).
 The real version of \eqref{main free boundary} has been studied earlier by
  Caffarelli-Nirenberg-Spruck \cite{CNS}, Guan \cite{Guan94}, Trudinger \cite{Tr95}...
 (see also \cite{ITW04}, \cite{Guan14}, \cite{GJ15}, \cite{GJ16}, \cite{JT20} for some recent developments).
 Some  important equations of the form \eqref{main free boundary} are
  the complex Monge-Amp\`ere equations,
 the complex Hessian equations and the complex Hessian quotient equations, where we take, respectively, $f(x)=(\sigma_n(x))^{1/n}$,
 $f(x)=(\sigma_k(x))^{1/k} (1\leq k\leq n)$ and $f(x)=(\sigma_k(x)/\sigma_l(x))^{k-l} (1\leq l<k\leq n)$.

 The viscosity method introduced in \cite{CL83, Lions83} (see \cite{CIL92} for a survey) is useful for studying
 partial differential equations in the non-smooth setting.
  A viscosity approach to the equation \eqref{main free boundary} has been studied in \cite{DDT}.
  The goal of this paper is to expand this research direction. Following \cite{CIL92}, a
  function $u\in USC(\Omega)$ is a viscosity subsolution (resp., supersolution) of \eqref{main free boundary} iff for every $z_0\in\Omega$ and for any
  $C^2$-smooth function $q$ in a neighbourhood $U$ of $z_0$ such that $(u-q)(z_0)=0=\max_U(u-q)$
  (resp., $(u-q)(z_0)=0=\min_U(u-q)$),
   we have $F(Hq(z_0))\geq\psi(z_0, u(z_0))$ (resp., $F(Hq(z_0))\geq\psi(z_0, u(z_0))$). 
   The reader can find more details about the defintions and  properties of
   viscosity sub/super-solutions in \cite{DDT}
  (see also  the Preliminaries). By \cite[Lemma 4.6, Remark 4.9 and Theorem B.8]{HL},
   for every $u\in USC(\Omega)$, the following conditions are equivalent:
   \begin{itemize}
   	\item[(i)] $F(Hu)\geq 0$ in the viscosity sense in $\Omega$.
   	\item[(ii)] For every open set $U\Subset\Omega$ there exists a decreasing sequence $\{u_j\}$ of smooth
   	 $\Gamma$-subharmonic functions on $U$ such that $u_j\rightarrow u$ as $j\rightarrow\infty$. Here a smooth
   	 function $w$ is  $\Gamma$-subharmonic if $Hw(z)\in M(\Gamma, n)$ for every $z$.
   \end{itemize}
We say that a function $u\in USC(\Omega)$ is $\Gamma$-subharmonic
if it satisfies the above equivalent conditions. Since $\Gamma\subset\Gamma_1$, every  $\Gamma$-subharmonic function is subharmonic. An alternative proof for the equivalence of (i) and (ii)
 is provided in this paper (see Corollary \ref{prop gsh gvis}). 
Futhermore, we generalize this fact 
for viscosity subsolutions of \eqref{main free boundary} in the case where 
$\psi (z, r)$ is independent of $r$. 
     Our first main result is as follows:
\begin{The}\label{main 1}
	Assume that $\psi (z, r)$ does not depend on the last variable $r\in\R$. Then a function $u\in USC(\Omega)$ is a viscosity subsolution of \eqref{main free boundary} iff $u$ is $\Gamma$-subharmonic
and $F(H(u\ast\chi_{\epsilon}))\geq \psi\ast\chi_{\epsilon}$ in $\Omega_{\epsilon}$ (in the classical sense).
Here $\chi_{\epsilon}$ is the standard modifier, 
$\ast$ is the convolution operator and 
$\Omega_{\epsilon}=\{z\in\Omega: d(z, \partial\Omega)>\epsilon\}$.
\end{The}
When $f$ satisfies some conditions on the limit at infinity, we use Theorem \ref{main 1} to show that 
every viscosity subsolution of \eqref{main free boundary} can be approximated by a decreasing sequence of
classical subsolution of \eqref{main free boundary}. In the cases of Monge-Amp\`ere equations and Hessian 
equations, this fact has been proved in \cite{EGZ} and \cite{Lu}.
\begin{Cor}\label{cor1 main1}
	Assume that $\psi (z, r)$ does not depend on $r$ and
	\begin{center}
		$\lim\limits_{R\to\infty}f(R, R,..., R)>\sup\limits_{\Omega}\psi.$
	\end{center} 
Then a function 
	$u\in USC(\Omega)$ is a viscosity subsolution of \eqref{main free boundary} iff for every 
	open set $U\Subset\Omega$, there exists a decreasing sequence $\{u_j\}$ of smooth
	$\Gamma$-subharmonic functions on $U$ such that $u_j\rightarrow u$ as $j\rightarrow\infty$ and 
	$F(Hu_j(z))\geq\psi (z)$ in $U$ for every $j$.
\end{Cor}
Our second purpose is to study the comparison principle for \eqref{main free boundary}.
 It follows from \cite{I89} that
the comparison principle holds if $\psi(z, r)-\epsilon r$ is non-decreasing in $r$ for some $\epsilon>0$.
Under appropriate growth restrictions on the behavior of $F$, one can permit $\epsilon=0$ (see \cite{DDT}).
In this paper, we establish a version of the comparison principle with weaker conditions for $F$:
\begin{The}\label{main 2}
		Let $\Omega\subset\C^n$ be a bounded domain. Let $u\in USC\cap L^{\infty}(\overline{\Omega})$ and 
	$v\in LSC\cap L^{\infty}(\overline{\Omega})$, respectively, be a bounded subsolution and a 
	bounded supersolution of the equation
	\begin{equation}
	F(H w)=\psi (z, w),
	\end{equation}  in $\Omega$. 
		Assume that $u\leq v$ in $\partial\Omega$ and
	\begin{center}
		$\lim\limits_{R\to\infty}f(R, R,..., R)>\sup\limits_{K}\psi (z, v(z)),$
	\end{center}
	for every $K\Subset\Omega$. Then $u\leq v$ in $\Omega$.
\end{The}
By Theorem \ref{main 2} and the Perron method, 
if $v$ is a viscosity supersolution of \eqref{main free boundary} satisfying some suitable conditions
then the function
\begin{center}
	$\Phi_v=\sup\{w: w$ is a subsolution of \eqref{main free boundary}, $w\leq v \},$
\end{center}
is a discontinuous viscosity solution of \eqref{main free boundary} (see Proposition \ref{prop perron}).
In the case where $\psi(z, r)$ is independent of $r$, we show that $\Phi_v$ is further locally approximated by 
 a decreasing sequence of viscosity solutions. 
 \begin{The}\label{main 3}
 	Assume that $\psi$ does not depend on the last variable and
 		\begin{center}
 			$\lim\limits_{R\to\infty}f(R, R,..., R)>\psi (z),$
 		\end{center}
 		for every $z\in\Omega$.
 		Suppose that $v\in L^{\infty}\cap LSC(\overline{\Omega})$ is a bounded viscosity supersolution 
 		of \eqref{main free boundary} in $\Omega$ which is continuous at every point $z\in\partial\Omega$.
 	Then, for every relatively compact open subset $U$ of $\Omega$, there exists
 	a decreasing sequence $u_j$ of viscosity solutions of \eqref{main free boundary} in $U$ such that
 	$\lim_{j\to\infty}u_j=\Phi_v$ in $U$.
 \end{The}
\begin{ackn}
	{\rm The authors would like to thank Lu Hoang Chinh for fruitful discussions on discontinuous viscosity solutions.
		This research began while the first named author was visiting Vietnam Institute for Advanced Study in Mathematics(VIASM). He would like to thank  
		the institution for the hospitality.}
\end{ackn}
\section{Preliminaries}
In this section, we recall the definitions and some properties of viscosity sub/super-solutions.
\begin{Def} (Test functions)
	Let  $ w : \Omega\longrightarrow \R$ be any function defined in $\Omega$ and $z_0 \in\Omega$ a given point. 
	An upper test function (resp., a lower test function) for $w$ at the point $z_0$ 
	is  a $C^2$-smooth function $q$ 
	in a neighbourhood of $z_0$ such that $  w (z_0) = q (z_0)$ and $w \leq q$ (resp., $w \geq q$) in a neighbourhood of $z_0$. 
\end{Def}

\begin{Def}\label{def vis}
	1. A  function  $ u \in USC(\Omega)$ is said to be a (viscosity) subsolution of
	\begin{equation}\label{Ellipticfree}
	F(Hu)=\psi(z, u),
	\end{equation}
	 in $\Omega$ if for any point $ z_0\in\Omega$ and any upper test function $q$ for $u$ at $z_0$, we have
	$F(Hq(z_0))  \geq \psi (z_0, u(z_0))$ (and then $Hq(z_0)\in \overline{M(\Gamma, n)}$).
	In this case, we also say that $F(Hu)\geq \psi (z, u)$ in the viscosity sense in $\Omega$.
	
	\medskip
	2. A  function  $ v \in LSC(\Omega)$ is said to be a (viscosity)
	supersolution of  \eqref{Ellipticfree}
	in $\Omega$ if  for any point 
	$ z_0 \in\Omega$ and any lower test  function $q$ for $v$ at $z_0$,
	 we have 
	$F(Hq(z_0))\leq\psi (z_0, u(z_0)).$
	In this case, we also say that $F(Hv)\leq \psi (z, v)$ in the viscosity sense in $\Omega$.\\
	\medskip
	3. A function $ u\in C(\Omega)$ is said to be a (viscosity) solution of  \eqref{Ellipticfree} in $\Omega$ if it is a subsolution and a supersolution of
	 \eqref{Ellipticfree} in $\Omega$.\\
	 	\medskip
	 4. A function $u\in L^{\infty}(\Omega)$ is said to be a discontinuous viscosity solution of  \eqref{Ellipticfree} in $\Omega$ if $u^*$ is a subsolution and $u_*$ is a supersolution of 
	 \eqref{Ellipticfree} in $\Omega$.
\end{Def}
It follows from the definition directly that
 if $u, v$ are viscosity subsolutions of \eqref{Ellipticfree} then $\max\{ u, v\}$
is a viscosity subsolution of \eqref{Ellipticfree}. Furthermore, we
also have:
\begin{Prop}\label{prop gluing}
	Assume that $ G\subsetneq \Omega$ is an open set. Suppose that $v$ is a viscosity subsolution of
	\eqref{Ellipticfree} in $G$ and $u$ is a viscosity subsolution of \eqref{Ellipticfree} in $\Omega$ such that
	\begin{center}
		$\limsup\limits_{G\ni z\to z_0}v(z)\leq u(z_0)$,
	\end{center}
for every $z_0\in \partial G\cap \Omega$. Then, the function
\begin{center}
	$\tilde{u}=\begin{cases}
	u\quad\mbox{in}\quad\Omega\setminus G,\\
	\max\{u, v\}\quad\mbox{in}\quad G,
	\end{cases}$
\end{center}
is a viscosity subsolution of \eqref{Ellipticfree} in $\Omega$.
\end{Prop}
In the case of Hessian equations (i.e., $f(x)=(\sigma_k(x))^{1/k}$ and $\Gamma=\Gamma_k$), we use the notation
$F_k$ instead of $F$. The following result has been proved in \cite{Lu}:
\begin{Prop}
Let $u\in USC(\Omega)$. Then the following conditions are equivalent:
\begin{itemize}
	\item [a)] $u$ is a viscosity subsolution of the equation $F_k(Hw)=\psi (z, w)$ in the sense of Definition \ref{def vis};
	\item [b)] for every $z_0\in\Omega$, for every upper test function $q$ for $u$ at $z_0$, we have
 $\sigma_k(\lambda(Hq(z_0))\geq \psi^k (z_0, u(z_0))$ (it does not require that
  $Hq(z_0)\in \overline{M(\Gamma_k, n)}$).
\end{itemize}
\end{Prop}
Actually, in \cite{Lu}, a function $u\in USC(\Omega)$ is called a viscosity subsolution of the equation
$\sigma_k(\lambda(Hw))=\psi^k (z, w)$ if it satisfies the condition b) in the above proposition. By the proof of \cite[Lemma 3.7]{Lu}, if b) is satisfied
 then, for every $z_0\in\Omega$ and for every upper test function $q$ for $u$ at $z_0$,
  the Hessian matrix $Hq(z_0)=(\frac{\partial^2q}{\partial z_{\alpha}\partial\overline{z_{\beta}}}(z_0))$ is $k$-positive, i.e.,
 	$Hq(z_0)\in  \overline{M(\Gamma_k, n)}$. Then $b) \Rightarrow a)$. The fact $a)\Rightarrow b)$ is obvious.\\
 
	If $u\in USC(\Omega)$ (resp., $u\in LSC(\Omega)$)
then for every $x\in\Omega$, the set
	\begin{center}
		$\Jj^{\pm}u(z)=\{( Dw(z), Hw(z))\in \R^{2n}\times\mathcal{H}^n: w$ is 
		an upper test function (resp. a lower test function) for $u$ at $z \}$
	\end{center}
	is called the super-(resp., sub-)differential of $u$ at $z$. The set
	\begin{center}
		$\overline{\Jj}^{\pm}u(z)=\{(p, Z)\in\R^{2n}\times\mathcal{H}^n: \exists z_m\rightarrow z$
		and $(p_m, Z_m)\in\Jj^{\pm}u(z_m)$ such that $(p_m, Z_m)\rightarrow (p, Z)$ and
		$u(z_m)\rightarrow u(z)  \}$
	\end{center}
is called the limiting	super-(resp. sub-)differential of $u$ at $z$.
By the continuity of $F$ and $\psi$, the limiting super/sub-differentials can be used to identify
 viscosity sub/super-solutions  as follows:
\begin{Prop}\label{prop def3 vis}
	a) Let $u\in USC(\Omega)$. Then $u$ is a viscosity subsolution of the equation
	\eqref{Ellipticfree} iff  for any point $x\in\Omega$, for every $(p, X)\in \overline{\Jj}^{+}u(z)$, we have
	$$
	F(X)  \geq\psi (z, u).
	$$
	b) Let  $u\in LSC(\Omega)$. Then $u$ is a viscosity  supersolution of the equation
	\eqref{Ellipticfree} iff  for any point $ x\in\Omega$, for every $(p, X)\in \overline{\Jj}^{-}u(z)$, we have
	$$
	F(X)  \leq \psi (z, u).
	$$
\end{Prop}
The following proposition is deduced by combining Proposition \ref{prop def3 vis} and
\cite[Proposition 4.3]{CIL92}:
\begin{Prop}\label{prop liminf limsup}
a) Assume that $\{u_{\alpha}\}$ is a family of viscosity subsolutions of the equation 
\begin{equation}\label{eq prop liminf limsup}
F(Hw)=\psi (z, w),
\end{equation}
in $\Omega$. If $u=\sup_{\alpha}u_{\alpha}$ is locally bounded from above then its usc regularization $u^*$
is a viscosity subsolution of \eqref{eq prop liminf limsup} in $\Omega$.\\
b) Assume that $\{u_{\alpha}\}$ is a family of viscosity supersolutions of \eqref{eq prop liminf limsup}
in $\Omega$. If $u=\inf_{\alpha}u_{\alpha}$ is locally bounded from below then its lsc regularization $u_*$
is a viscosity supersolution of \eqref{eq prop liminf limsup} in $\Omega$.\\
b)  Assume that $u_j$ is a decreasing (resp., increasing)
 sequence of viscosity subsolutions (resp., supersolutions) to \eqref{eq prop liminf limsup}. 
Then $u=\lim_{j\to\infty}u_j$ is either a viscosity subsolutions (resp., supersolutions)
 to \eqref{eq prop liminf limsup} or
identically $-\infty$ (resp., $\infty$).
\end{Prop}
\section{Approximation of subsolutions}
In this section, we will prove the Theorem \ref{main 1} and Corollary \ref{cor1 main1}.
 First, we have the following lemma:
\begin{Lem}\label{lem vis sh}
	There exists a mapping
	\begin{center}
		$\Phi:  M(\Gamma, n)\rightarrow M(\Gamma_n, n)$\\
		$H\mapsto\widetilde{H}=\Phi (H)$	
	\end{center}
	depending on $F$ such that 
	\begin{itemize}
		\item [a)] For all $B\in \overline{M(\Gamma, n)}$,
		\begin{equation}
		F(B)=\inf\{\Delta_{\widetilde{H}}B+F(H)-\Delta_{\widetilde{H}}H: H\in M(\Gamma, n) \},
		\end{equation}
		where $\Delta_{\widetilde{H}}B=trace(\widetilde{H}B)=\sum\limits_{j,k=1}^n\widetilde{h}_{jk}b_{kj}$.
		\item [b)] For all $B\in \mathcal{H}^n$, if
		\begin{equation}\label{eq0 lem vis sh}
		\inf\{\Delta_{\widetilde{H}}B+F(H)-\Delta_{\widetilde{H}}H: H\in M(\Gamma, n) \}\geq 0,
		\end{equation}
		then $B\in\overline{M(\Gamma, n)}$.
	\end{itemize}
\end{Lem}
\begin{proof}
	a) By the concavity of $F$ in $ M(\Gamma, n)$, for every $H\in M(\Gamma, n)$, the subdifferential
	 $\partial (-F(H))$ is nonempty, i.e., there exists $\tilde{H}\in\mathcal{H}^n\setminus\{0\}$ such that
	\begin{equation}\label{eq1 lem vis sh}
	F(B)-F(H)\leq \Delta_{\tilde{H}}(B-H)=\Delta_{\tilde{H}}(B)-\Delta_{\tilde{H}}(H),
	\end{equation}
	for all $B\in \overline{M(\Gamma, n)}$. Moreover,
	\begin{equation}\label{eq1' lem vis sh}
		F(B)=\lim\limits_{H\rightarrow B}(F(H)+\Delta_{\tilde{H}}(B-H)).
	\end{equation}
	Combining \eqref{eq1 lem vis sh} and \eqref{eq1' lem vis sh}, we have
	\begin{equation}\label{eq2 lem vis sh}
	F(B)=\inf\{\Delta_{\tilde{H}}(B)+F(H)-\Delta_{\tilde{H}}(H): H\in M(\Gamma, n) \}.
	\end{equation}
	By \eqref{F is increasing} and \eqref{eq1 lem vis sh}, for every  $N\in \overline{M(\Gamma_n, n)}\setminus\{0\}$ and for each $H\in  M(\Gamma, n)$, we have
	\begin{equation}\label{eq3 lem vis sh}
	\Delta_{\tilde{H}}(N)=\Delta_{\tilde{H}}(H+N-H)\geq F(H+N)-F(H)>0.
	\end{equation}
	Hence,  $\widetilde{H}\in M(\Gamma_n, n)$ for every $H\in  M(\Gamma, n)$.\\
	b) Assume that $B\notin\overline{M(\Gamma, n)}$ and the condition \eqref{eq0 lem vis sh} is satisfied.
	Let $t_0>0$ such that $B+t_0I\in\partial M(\Gamma, n)$. Then, for every $t>t_0$, 
	we have $B+tI\in M(\Gamma, n)$. By the assumption, we get
	\begin{center}
	$\Delta_{\widetilde{B+tI}}B+F(B+tI)-\Delta_{\widetilde{B+tI}}(B+tI)\geq
	\inf\limits_{H\in M(\Gamma, n)}\{\Delta_{\widetilde{H}}B+F(H)-\Delta_{\widetilde{H}}H\}\geq 0, $	
	\end{center}
for every $t>t_0$. Then
	\begin{center}
		$F(B+tI)-t\Delta_{\widetilde{B+tI}}I\geq 0,$
	\end{center}
	for every $t>t_0$. Letting $t\searrow t_0$, we get
	\begin{equation}\label{eq4 lem vis sh}
	\limsup\limits_{t\to t_0^+}\Delta_{\widetilde{B+tI}}I\leq \lim\limits_{t\to t_0^+}\dfrac{F(B+tI)}{t}
	=\dfrac{F(B+t_0I)}{t_0}=0.
	\end{equation}
	Moreover, it follows from \eqref{eq1 lem vis sh} that
\begin{center}
	$F(B+(1+t_0)I)-F(B+tI)\leq (1+t_0-t)\Delta_{\widetilde{B+tI}}I,$
\end{center}
for every $t_0<t<t_0+1$. Letting $t\searrow t_0$, we get
\begin{equation}\label{eq5 lem vis sh}
	\liminf\limits_{t\to t_0^+}\Delta_{\widetilde{B+tI}}I\geq
	 \lim\limits_{t\to t_0^+}\dfrac{F(B+(1+t_0)I)-F(B+tI)}{1+t_0-t}=F(B+(1+t_0)I)>0.
\end{equation}
	By \eqref{eq4 lem vis sh} and \eqref{eq5 lem vis sh}, we get a contradiction.

	Thus, the condition \eqref{eq0 lem vis sh} implies that $B\in\overline{M(\Gamma, n)}$.
\end{proof}
\begin{Cor}\label{cor sum of vis sub}
	Let $t\in [0, 1]$ and $0\leq\psi_1, \psi_2\in C(\Omega)$. Assume that $u_j$
	is a viscosity subsolution of the equation $F(Hw)=\psi_j(z)$ in $\Omega$ for $j=1, 2$. Then,
	the function $tu_1+(1-t)u_2$ is a 
	subsolution of the equation $F(Hw)=t\psi_1(z)+(1-t)\psi_2(z)$.
\end{Cor}
\begin{proof}
	By Lemma \ref{lem vis sh}, 	we have, for $j=1,2$,
	\begin{equation}\label{eq1 cor sum of vis sub}
	\Delta_{\widetilde{H}}u_j+F(H)-\Delta_{\widetilde{H}}H\geq \psi_j,
	\end{equation}
	in the viscosity sense in $\Omega$ for every $H\in M(\Gamma, n)$. Then, it follows from \cite[Proposition 3.2.10', page 147]{Hor94} that \eqref{eq1 cor sum of vis sub} holds in the distribution sense.
	Therefore, we have
	\begin{equation}\label{eq2 cor sum of vis sub}
	\Delta_{\widetilde{H}}(tu_1+(1-t)u_2)+F(H)-\Delta_{\widetilde{H}}H\geq t\psi_1(z)+(1-t)\psi_2(z),
	\end{equation}
	in the distribution sense. Using again \cite[Proposition 3.2.10', page 147]{Hor94}, we get
	\eqref{eq2 cor sum of vis sub} holds in the viscosity sense. Thus, by Lemma \ref{lem vis sh}, we obtain
	\begin{center}
		$F(H(tu_1+(1-t)u_2))\geq t\psi_1(z)+(1-t)\psi_2(z),$
	\end{center}
	in the viscosity sense.
\end{proof}
\begin{Cor}\label{prop gsh gvis}
	Let $u\in USC(\Omega)$. Then the following conditions are equivalent
	\begin{itemize}
		\item [a)] $u$ is subharmonic and for every $\epsilon>0$,
		$u\ast\chi_{\epsilon}$ is $\Gamma$-subharmonic in $\Omega_{\epsilon}$.
		Here $\chi_{\epsilon}$ is the standard modifier, 
		$\ast$ is the convolution operator and 
		$\Omega_{\epsilon}=\{z\in\Omega: d(z, \partial\Omega)>\epsilon\}$.
		\item [b)]  $u$ is $\Gamma$-subharmonic.
		\item [c)] $F(Hu)\geq 0$ in the viscosity sense, i.e., for any point $ z_0\in \Omega$ and any upper test function $q$ for $u$ at $z_0$, we have $Hq(z_0)\in \overline{M(\Gamma, n)}$.
	\end{itemize} 
\end{Cor}
\begin{proof}
	$(a)\Rightarrow(b)$ and $(b)\Rightarrow (c)$ are clear. It remains to show $(c)\Rightarrow (a)$.
	
	Assume that $F(Hu)\geq 0$ in the viscosity sense. Then, by the definition and by the condition
	 $\Gamma\subset\Gamma_1$, we have $\Delta u\geq 0$ in the
	viscosity sense. Hence, by \cite[Proposition 3.2.10', page 147]{Hor94}, we get $u\in SH(\Omega)$.
	
	Moreover, it follows from Lemma \ref{lem vis sh} that
	\begin{equation}\label{eq1 prop vis sh}
	\Delta_{\widetilde{H}}u+F(H)-\Delta_{\widetilde{H}}H\geq 0,
	\end{equation}
	in the viscosity sense for every $H\in M(\Gamma, n)$. Then, it follows from Lemma \ref{lem poisson} that 
	\eqref{eq1 prop vis sh} holds in the distribution sense. Hence
	\begin{center}
		$\Delta_{\widetilde{H}}(u\ast\chi_{\epsilon})+F(H)-\Delta_{\widetilde{H}}H\geq 0,$
	\end{center}
	in the classical sense in $\Omega_{\epsilon}$ for every $H\in M(\Gamma, n)$. Using again Lemma \ref{lem vis sh}, we have 
	$H(u\ast\chi_{\epsilon})(z)\in \overline{M(\Gamma, n)}$ for every $\epsilon>0$ and $z\in\Omega_{\epsilon}$. Thus  $u\ast\chi_{\epsilon}$ is
	$\Gamma$-subharmonic in $\Omega_{\epsilon}$.
	
	The proof is completed.
\end{proof}
\begin{Lem}\label{lem poisson}
	Let $U\subset\R^N (N\geq 2)$ be a bounded domain. Assume that $g\in C(U)$ and $u\in USC(U)$. Then, $\Delta u\geq g$ in the
	viscosity sense iff $\Delta u\geq g$ in the distribution sense.
\end{Lem}
\begin{proof}
	If $g\in C_c^2(\R^N)$ then
	 $v=E\ast g$ is a classical solution to the Poisson equation $\Delta w=g$, where
	\begin{center}
		$E(x)=\begin{cases}
		\frac{1}{2\pi}\log|x|\qquad (N=2),\\
		\frac{-1}{N(N-2) c_N}\frac{1}{|x|^{N-2}}\qquad (N\geq 3),
		\end{cases}$
	\end{center}
	and $c_N$ is the volume of $\B^N$. It follows from \cite[Proposition 3.2.10', page 147]{Hor94} that $\Delta (u-v)\geq 0$ in the
	viscosity sense iff $\Delta (u-v)\geq 0$ in the distribution sense. Hence, 
	$\Delta u\geq g (=\Delta v)$ in the
	viscosity sense iff $\Delta u\geq g$ in the distribution sense.
	
	In the general case, since the problem is local, we can assume that $g\in C_c(\R^N)$.
	Then, we can choose a sequence $g_j\in C_c^2(\R^N)$ such that $g_j\nearrow g$ as $j\rightarrow\infty$. Hence,
	by the above argument, we have
	\begin{flushleft}
		$\begin{array}{ll}
		(\Delta u\geq g \mbox{ in the viscosity sense}) &\Leftrightarrow (\Delta u\geq g_j \mbox{ in the viscosity sense for every j})\\
		&\Leftrightarrow (\Delta u\geq g_j \mbox{ in the distribution sense for every j})\\
		&\Leftrightarrow (\Delta u\geq g \mbox{ in the distribution sense}).
		\end{array}$
	\end{flushleft}
\end{proof}
\begin{proof}[Proof of Theorem \ref{main 1}]
If $u$ is $\Gamma$-subharmonic
and $F(H(u\ast\chi_{\epsilon}))\geq \psi\ast\chi_{\epsilon}$ in $\Omega_{\epsilon}$ for every $\epsilon>0$ 
then, by the definition, we have $F(H(u\ast\chi_{\epsilon}))\geq \psi\ast\chi_{\epsilon}$ in the viscosity sense in
 $\Omega_{\epsilon}$ for every $\epsilon>0$. Hence,
 $F(H(u\ast\chi_{\epsilon}))\geq \psi_r$ in the viscosity sense in
 $\Omega_r$ for every $0<\epsilon<r$, where $\psi_r(z)=\inf_{|z-w|<r}\psi(w)$. Since $u$ is subharmonic, we have 
 $u\ast\chi_{\epsilon}\searrow u$ as $\epsilon\searrow 0$. Using Proposition \ref{prop liminf limsup}, we get
 $F(H u)\geq \psi_r$ in the viscosity sense in
 $\Omega_r$ for every $r>0$. Letting $r\searrow 0$, we obtain $F(H u)\geq \psi$ in the viscosity sense in
 $\Omega$.

Conversely, assume that $u$ is a viscosity subsolution of the equation
	$F(Hw)=\psi(z)$
in $\Omega$. By  Corollary \ref{prop gsh gvis}, we have $u$ and
 $u\ast\chi_{\epsilon}$ are $\Gamma$-subharmonic ($\epsilon>0$). Moreover, by the same argument as in the proof of 
 Corollary \ref{prop gsh gvis}, we also have
 \begin{center}
 		$\Delta_{\widetilde{H}}(u\ast\chi_{\epsilon})+F(H)-\Delta_{\widetilde{H}}H\geq \psi\ast\chi_{\epsilon},$
 \end{center}
in the classical sense in $\Omega_{\epsilon}$ for every $H\in M(\Gamma, n)$.
Hence, it follows from  Lemma \ref{lem vis sh} that  
$F(H(u\ast\chi_{\epsilon}))\geq \psi\ast\chi_{\epsilon}$ in $\Omega_{\epsilon}$ in the classical sense.
\end{proof}
\begin{proof}[Proof of Corollary \ref{cor1 main1}]
	 If there exists a decreasing sequence $\{u_j\}$ of smooth
	$\Gamma$-subharmonic functions on $U$ such that $u_j\rightarrow u$ as $j\rightarrow\infty$ and 
	$F(Hu_j(x))\geq\psi (z)$ in $\Omega$ for every $j$ then, by Proposition \ref{prop liminf limsup},
	$u$ is a viscosity subsolution of \eqref{main free boundary}.
	
	For the converse, assume that $u$ is a viscosity subsolution of \eqref{main free boundary} and
	$U$ is a relatively compact open subset of $\Omega$. By Theorem \ref{main 1}, we have
	$u\ast \chi_{\epsilon}\searrow u$ as $\epsilon\searrow 0$ and $F(Hu\ast \chi_{\epsilon})\geq\psi\ast\chi_{\epsilon}$ in $U$ for every $0<\epsilon<\epsilon_0$,
	where $\epsilon_0>0$ is small enough such that $U\Subset\Omega_{\epsilon_0}$. Since $\psi\ast\chi_{\epsilon}$
	converges uniformly to $\psi$ in $U$, there exists $0<...<\epsilon_{j+1}<\epsilon_j<...<\epsilon_1<\epsilon_0$
	such that $\lim_{j\to\infty}\epsilon_j=0$ and 
	\begin{center}
	$|\psi\ast\chi_{\epsilon}-\psi|<\dfrac{1}{2^j},$
	\end{center}
	in $U$. By the assumption, there exists $R\gg 1$ such that $f(R, R, ...,R)>\sup_{U}\psi+r$ for some $0<r<1$.
	For every $j>-\log_2 r$, we denote:
	\begin{center}
		$u_j(z)=u\ast \chi_{\epsilon_j}+\dfrac{R|z|^2}{2^{j+1}r}.$
	\end{center}
Then, $u_j$ is a decreasing sequence of smooth $\Gamma$-subharmonic functions in $U$ satisfying
$\lim_{j\to\infty}u_j=u$. Moreover, for every $j>-\log_2 r$ and for each $z\in U$, we have
\begin{flushleft}
	$\begin{array}{ll}
	F(Hu_j)&\geq (1-\dfrac{1}{2^{j}r})F(Hu\ast \chi_{\epsilon_j})
	+\dfrac{1}{2^jr}F(H(u\ast \chi_{\epsilon_j}+\dfrac{R|z|^2}{2}))\\
	&\geq (1-\dfrac{1}{2^{j}r})\psi\ast\chi_{\epsilon}+\dfrac{1}{2^jr}F(R I)\\
	&\geq (1-\dfrac{1}{2^{j}r})(\psi(z)-\dfrac{1}{2^j})+\dfrac{1}{2^jr}(\psi(z)+r)\\
	&>\psi(z),
	\end{array}$
\end{flushleft}
in $U$. 
\end{proof}
\section{Comparison principle and applications}
Now we prove the second main theorem of this paper:
\begin{The}
 Let $u\in USC\cap L^{\infty}(\overline{\Omega})$ and 
$v\in LSC\cap L^{\infty}(\overline{\Omega})$, respectively, be a bounded subsolution and a 
bounded supersolution of the equation
\begin{equation}\label{eq0 elliptic compa}
F(H w)=\psi (z, w),
\end{equation}  in $\Omega$. 
Assume that $u\leq v$ in $\partial\Omega$ and
\begin{equation}\label{eq f at infty}
	\lim\limits_{R\to\infty}f(R, R,..., R)>\sup\limits_{\Omega}\psi (z, v(z)).
\end{equation}
Then $u\leq v$ in $\Omega$.
\end{The}
\begin{proof}
	First, we consider the case where $u-\delta |z|^2$ is a subsolution of \eqref{eq0 elliptic compa} for
	some $\delta>0$. Assume that there exists $z_0\in\Omega$ such that
	\begin{equation}\label{eq1 elliptic compa}
	(u-v)(z_0)=\max\limits_{\overline{\Omega}}(u-v)>0.
	\end{equation}
	For each $N>0$, we denote
	\begin{center}
		$\phi_N(z, w)=u(z)-v(w)-N|z-w|^2,$
	\end{center}
for all $(z, w)\in\overline{\Omega}^2$. Since $\overline{\Omega}^2$ is compact and $\phi_N$ is upper semicontinuous, there exists $(z_N, w_N)\in \overline{\Omega}^2$ such that
\begin{center}
	$\phi_{N}(z_N, w_N)=\max\limits_{\overline{\Omega}^2}\phi_N$.
\end{center}
Moreover, by \cite[Lemma 3.1]{CIL92}, we can assume that $z_N$ and $w_N$ converge to $z_0$
as $N\rightarrow\infty$. In particular, there exists $N_0>0$ such that $z_N, w_N\in B(z_0, R)$ for every $N>N_0$,
where $0<R<d(z_0, \partial\Omega)$.
By the maximum principle \cite[Theorem 3.2]{CIL92}, there exist
 $Z_N, W_N\in\mathcal{H}^n$ such that 
 $(2N(z_N-w_N), Z_N)\in \overline{\Jj}^+u(z_N)$, $(2N(z_N-w_N), W_N)\in \overline{\Jj}^-v(w_N)$ and
 $Z_N\leq W_N$ for all $N>N_0$. Hence, we have
 \begin{equation}\label{eq2 elliptic compa}
 F(Z_N-2\delta I)\geq \psi (z_N, u(z_N)),
 \end{equation}
 and 
 \begin{equation}\label{eq3 elliptic compa}
 F(W_N)\leq \psi (w_N, v(w_N)),
 \end{equation}
 and
 \begin{equation}\label{eq4 elliptic compa}
 F(Z_N)\leq F(W_N).
 \end{equation}
 Combining \eqref{eq3 elliptic compa} and \eqref{eq4 elliptic compa}, we get
 \begin{equation}\label{eq5 elliptic compa}
 F(Z_N)\leq M,
 \end{equation}
 for all $N>N_0$, where $M:=\sup\{\psi (z, v(z)): x\in B(z_0, R)\}$.
 Since
 $\lim_{R\to\infty}f(R,..., R)>M$, there exist $R_0\gg 1$ and $0<r\ll 1$ such that
 \begin{center}
 $F(R I)=f(R,...,R)>M+r.$
 \end{center}
for all $R>R_0$.
Then, by \eqref{eq5 elliptic compa} and by the concavity of $F$, we have, for every $0<\epsilon<1$,
\begin{flushleft}
	$\begin{array}{ll}
	\dfrac{F(Z_N)-F(Z_N-2\delta I)}{2\delta}&\geq \dfrac{F(Z_N+R_0I)-F(Z_N)}{R_0}\\
	&\geq \dfrac{\epsilon F(Z_N/\epsilon)+(1-\epsilon)F(R_0I/(1-\epsilon))-M}{R_0}\\
	&\geq\dfrac{(1-\epsilon)(M+r)-M}{R_0},
	\end{array}$
\end{flushleft}
for all $N>N_0$. Letting $\epsilon=\dfrac{r}{2M+r}$, we get
\begin{equation}\label{eq6 elliptic compa}
F(Z_N)\geq F(Z_N-\delta I)+\dfrac{2M\delta^2}{(2M+r)R_0},
\end{equation}
for every $N>N_0$. Combining \eqref{eq2 elliptic compa}, \eqref{eq3 elliptic compa}, \eqref{eq4 elliptic compa} 
and \eqref{eq6 elliptic compa}, we obtain
\begin{equation}\label{eq7 elliptic compa}
\psi (z_N, u(z_N))+\dfrac{2M\delta^2}{(2M+r)R_0}\leq \psi (w_N, v(w_N)),
\end{equation}
for every $N>N_0$. Since $z_N, w_N\rightarrow z_0$ and $\psi$ is uniformly continuous, we also have
\begin{equation}\label{eq8 elliptic compa}
\lim\limits_{N\to\infty}(\psi (z_N, u(z_N))-\psi (w_N, u(z_N)))=0.
\end{equation}
Moreover, it follows from \cite[Lemma 3.1]{CIL92} that $\lim_{N\to\infty}(u(z_N)-v(w_N))=\max\limits_{\overline{\Omega}}(u-v)>0$. Hence, since
$\psi$ is non decreasing in the last variable, we get
\begin{equation}\label{eq9 elliptic compa}
\liminf\limits_{N\to\infty}(\psi (w_N, u(z_N))-\psi (w_N, v(w_N)))\geq 0.
\end{equation}
Combining \eqref{eq7 elliptic compa}, \eqref{eq8 elliptic compa} and \eqref{eq9 elliptic compa}, we get
\begin{center}
	$0\leq \liminf\limits_{N\to\infty}(\psi (z_N, u(z_N))-\psi (w_N, v(w_N)))\leq -\dfrac{2M\delta^2}{(2M+r)R_0},$
\end{center}
and this is a contradiction. Thus, \eqref{eq1 elliptic compa} is not true.

In the general case, for each $\delta>0$, we denote $u_{\delta}(z)=u(z)+\delta(|z|^2-A)$, where $A=\max\{|w|^2: w\in\overline{\Omega} \}$.
By the above argument, we have $u_{\delta}\leq v$ in $\Omega$ for all $\delta>0$. Letting $\delta\searrow 0$, we get $u\leq v$ in $\Omega$.

The proof is completed.
\end{proof}
By using the Perron method \cite{CIL92} and Theorem \ref{main 2}, we obtain the following result:
\begin{Prop}\label{prop perron}
	Let 
	$v\in LSC\cap L^{\infty}(\overline{\Omega})$ be a 
	bounded supersolution of the equation
	\begin{equation}\label{eq0 perron}
	F(H w)=\psi (z, w),
	\end{equation}  in $\Omega$ such that 
	\begin{center}
		$\lim\limits_{\Omega\ni z\to \hat{z}}v(z)=v(\hat{z}),$
	\end{center}
	for all $\hat{z}\in\partial\Omega$ and
	\begin{center}
		$\lim\limits_{R\to\infty}f(R, R,..., R)>\sup\limits_{\Omega}\psi (z, v(z)).$
	\end{center}
 Denote by $S$ the set of all viscosity subsolutions $w$  to \eqref{eq0 perron} satisfying $w\leq v$.
 Then the function 
 \begin{center}
 		$u(z)=\sup\{w(z): w\in S \},$
 \end{center}
is a discontinuous viscosity solution of \eqref{eq0 perron} with $u=u^*\in S$.
\end{Prop}
\begin{proof}
	By Proposition 
	\ref{prop liminf limsup}, we have $u^*$ is a viscosity subsolution of the equation $F(D^2w)=\psi (x, w)$
	in $\Omega$. Moreover, since $u\leq v$ in $\Omega$, we have
	\begin{center}
		$\limsup\limits_{\Omega\ni z\to \hat{z}}u^*(z)=\limsup\limits_{\Omega\ni z\to \hat{z}}u(z)\leq
		\lim\limits_{\Omega\ni z\to \hat{z}}v(z)=v(\hat{z}),$	
	\end{center}
	for all $\hat{z}\in\partial\Omega$. Then, it follows from Theorem \ref{main 3} that $u^*\leq v$. Hence, $u^*\in S$ and $u=u^*$.
	 It remains to show that $u_*$ is a viscosity supersolution.
	
	Assume that  there exist a point
	$z_0\in\Omega$, an open neighbourhood $U\subset\Omega$ of $z_0$ and a function $\eta\in C^2(U)$ such that
	$\eta(z_0)=u_*(z_0)$, $\eta\leq u_*|_U$, $H\eta(z_0)\in M(\Gamma, n)$
	and 
	\begin{center}
		$F(H\eta (z_0))>\psi(z_0, \eta(z_0))$.
	\end{center}
	By the continuity of $F, \psi, \eta$ and $H\eta$, there exist
	$r, s>0$ such that $\overline{B(z_0, r)}\subset U$, $H\eta(x)-2sI\in M(\Gamma, n)$ for all 
	$z\in B(z_0, r)$ and
	\begin{center}
		$F(H\eta(z)-2sI)>\psi(z, \eta(z)+s),$
	\end{center}
	for every $z\in B(z_0, r)$. Denote
	\begin{center}
		$\tilde{\eta}(z)=\eta (z)-s|z-z_0|^2+\min\{s, \dfrac{sr^2}{4}\}$.
	\end{center}
	We have
	\begin{equation}
	F(H\tilde{\eta}(z))\geq \psi(z, \tilde{\eta}(z)), \quad\forall |z-z_0|\leq r,
	\end{equation}
and
\begin{equation}
\tilde{\eta}(z)\leq u(z), \quad\forall r/2\leq |z-z_0|\leq r.
\end{equation}
Denote
	\begin{center}
		$\tilde{u}(z)=\begin{cases}
		u(z)\quad\mbox{if}\quad z\in\Omega\setminus B(z_0, r),\\
		\max\{u(z),\tilde{\eta}(z) \}\quad\mbox{if}\quad z\in  B(z_0, r).
		\end{cases}$
	\end{center}
	Then $\tilde{u}\in S$ and $\tilde{u}\geq u$. Since $u=\sup\{w: w\in S \}$, we have $\tilde{u}=u$. Moreover, 
	\begin{center}
		$\tilde{u}_*(z_0)\geq\tilde{\eta}(z_0)\geq u_*(z_0)+\min\{s, \dfrac{sr^2}{4}\}>u_*(z_0).$
	\end{center}
	and it implies that $\tilde{u}$ is not identical to $u$. We get a contradiction. Thus, 
	$u_*$ is a supersolution of \eqref{eq0 perron}.
\end{proof}
Note that every harmonic function is a supersolution of \eqref{eq0 perron}. 
By using Theorem \ref{main 2} and Proposition \ref{prop perron}, we obtain the following result which will
be used in the next section:
\begin{Prop}\label{prop existence}
Assume that $\Omega$ is a bounded smooth domain and $\varphi$ is a continuous function on $\partial\Omega$
satisfying
\begin{center}
		$\lim\limits_{R\to\infty}f(R, R,..., R)>\psi (z, \sup\limits_{\partial\Omega}\varphi),$
\end{center}
for every $z\in\Omega$. Suppose that there exists $\underline{u}\in USC(\overline{\Omega})$ such that
 $\underline{u}|_{\partial\Omega}=\varphi$ and $F(H\underline{u})\geq\psi(z, \underline{u})$ in the 
 viscosity sense in $\Omega$. Then, there exists a unique $u\in C(\overline{\Omega})$ such that
 $u|_{\partial\Omega}=\varphi$ and $F(Hu)=\psi(z, u)$ in the  viscosity sense in $\Omega$.
\end{Prop}
\section{Maximal viscosity subsolutions}
In this section, we study some properties of maximal viscosity subsolutions (see below for the defintion).
 Theorem \ref{main 3}
is deduced by combining Proposition \ref{prop maximal} and Theorem \ref{the maximal}.

Similar to the concept of maximal plurisubharmonic functions \cite{Sad81} (see also \cite{Kli91}), 
we say that a viscosity subsolution $u$ for \eqref{main free boundary} is maximal
if $u$ satisfies the following condition: For every open set 
$U\Subset\Omega$ and for each $v\in USC(\overline{U})$ such that $v$ is a subsolution
for \eqref{main free boundary} in $U$ and $v\leq u$ in $\partial U$, we have $v\leq u$ in $U$.
\begin{Prop}\label{prop maximal}
	Under the assumption of Theorem \ref{main 3}, the function $\Phi_v$ is a maximal viscosity subsolution 
	for \eqref{main free boundary}.
\end{Prop}
\begin{proof}
	By Proposition \ref{prop perron}, we have $\Phi_v$ is a viscosity subsolution of
	\eqref{main free boundary}. We will show that $\Phi_v$ is maximal.
	
	Let $U$ be a relatively open subset of $\Omega$. Let 
	$w\in USC(\overline{U})$ such that $w$ is a subsolution
	for \eqref{main free boundary} in $U$ and $w\leq \Phi_v$ in $\partial U$. By Proposition
	\ref{prop gluing}, the function
	\begin{center}
		$u(z)=\begin{cases}
		\Phi_v(z)\quad\mbox{if}\quad z\in \Omega\setminus U,\\
		\max\{w(z), \Phi_v(z)\}\quad\mbox{if}\quad z\in U,
		\end{cases}$
	\end{center}
is a subsolution of \eqref{main free boundary} in $\Omega$. Since $u=\Phi_v\leq v$ in $\Omega\setminus U$,
it follows from Theorem \ref{main 2} that $u\leq v$ in $\Omega$. Then, by the definition of $\Phi_v$, we get
$u\leq\Phi_v$ in $\Omega$. 

 Thus $\Phi_v$ is a maximal viscosity subsolution of \eqref{main free boundary}.
\end{proof}
 \begin{The}\label{the maximal}
 	Assume that $\psi$ does not depend on the last variable
 	and
 	\begin{equation}\label{eq main3}
 	\lim\limits_{R\to\infty}f(R, R,..., R)>\psi (z),
 	\end{equation}
 	for every $z\in\Omega$.
	Suppose that $u$ is a maximal viscosity subsolution
	for \eqref{main free boundary} in $\Omega$.
	Then, for every relatively compact open subset $U$ of $\Omega$, there exists
	a decreasing sequence $u_j$ of viscosity solutions of \eqref{main free boundary} in $U$ such that
	$\lim_{j\to\infty}u_j=u$ in $U$.
\end{The}
In order to prove Theorem \ref{the maximal}, we need the following lemma:
\begin{Lem}\label{lem smooth}
For every $\epsilon>0$, there exists an open set $U$ with smooth boundary such that 
$\Omega_{\epsilon}\Subset U\Subset\Omega$, where
\begin{center}
	$\Omega_{\epsilon}=\{z\in\Omega: d(z, \partial\Omega)>\epsilon \}.$
\end{center}
\end{Lem}
\begin{proof}
	Consider the function $g(z)=(d\ast\chi_{\epsilon/4})(z)$, where
 $\chi_{\epsilon/4}$ is the standard modifier, 
	$\ast$ is the convolution operator and $d(z)=-d(z, \partial\Omega).$ We have $g$ is well-defined and smooth
	in $\Omega_{\epsilon/4}$. Moreover, for every $\epsilon/2<t<3\epsilon/4$,
	\begin{center}
	$\Omega_{\epsilon}\Subset U_t\Subset\Omega_{\epsilon/4}$,
	\end{center}
where $U_t=\{z\in\Omega_{\epsilon/4}: g(z)<-t\}$. In particular, we have
	$\partial U_t\subseteq g^{-1}(t)\Subset\Omega_{\epsilon/4}$ for every $\epsilon/2<t<3\epsilon/4$. 
	By Sard's Theorem, there exists $t_0\in (\epsilon/2, 3\epsilon/4)$ such that $Dg(z)\neq 0$ for every
	$z\in g^{-1}(t_0)$. Then, $\partial U_{t_0}= g^{-1}(t_0)$ and $U:=U_{t_0}$ is a smooth open set satisfying 
	$\Omega_{\epsilon}\Subset U\Subset\Omega$.
	
	The proof is completed.
\end{proof}
\begin{proof}[Proof of Theorem \ref{the maximal}]
	 By Lemma \ref{lem smooth}, there 
	exists a smooth open set $V$ such that $U\Subset V\Subset\Omega$. By the compactness of $\overline{U}$,
	we can assume that $V$ has finite (open) connected components. Then the problem is reduced to the case where
	$U$ is a smooth domain.

 By Corollary \ref{cor1 main1}, for every open neighbourhood $W\Subset\Omega$ of $\overline{U}$, there exists
  a decreasing sequence $\{v_j\}$ of smooth
  $\Gamma$-subharmonic functions on $W$ such that $v_j\rightarrow u$ as $j\rightarrow\infty$ and 
  $F(Hv_j(z))\geq\psi (z)$ in $W$ for every $j$.
  By Proposition \ref{prop existence}, for each $j\in\Z^+$, there exists
 a unique $u_j\in C(\overline{U})$ such that
$u_j|_{\partial U}=v_j|_{\partial U}$ and $F(Hu_j)=\psi(z)$ in the  viscosity sense in $U$.
We will show that $u_j$ decreases to $u$ as $j\rightarrow\infty$.

It follows from Theorem \ref{main 2} that $u_j\geq u_{j+1}\geq u$, and then
\begin{equation}\label{eq1 proof main 3}
\tilde{u}:=\lim_{j\to\infty}u_j\geq u,
\end{equation}
in $U$. It follows from Proposition \ref{prop liminf limsup} that $\tilde{u}$ is a viscosity subsolution of the equation $F(Hw)=\psi(z)$.  Moreover, we have
\begin{center}
$\tilde{u}|_{\partial U}=\lim\limits_{j\to\infty}u_j|_{\partial U}
=\lim\limits_{j\to\infty}v_j|_{\partial U}
=u|_{\partial U}.$
\end{center}
Since $u$ is a maximal viscosity subsolution of the equation $F(Hw)=\psi(z, w)$,
 we get
\begin{equation}\label{eq2 proof main 3}
u\geq \tilde{u},
\end{equation}
in $U$. Combining \eqref{eq1 proof main 3} and \eqref{eq2 proof main 3}, we obtain 
\begin{center}
	$u=\tilde{u}=\lim\limits_{j\to\infty}u_j,$
\end{center}
in $U$.

The proof is completed.
\end{proof}
\begin{Cor}
Under the assumption of Theorem \ref{the maximal}, if $u$ is bounded then $u$ is also a discontinuous 
viscosity solution of \eqref{main free boundary}.
\end{Cor}
\begin{proof}
By Theorem \ref{the maximal}, for every relatively compact open subset $U$ of $\Omega$, there exists
a decreasing sequence $u_j$ of viscosity solutions of \eqref{main free boundary} in $U$ such that
$\lim_{j\to\infty}u_j=u$ in $U$. Then, by Proposition \ref{prop liminf limsup}, $u_*=(\inf_ju_j)_*$
is a viscosity supersolution of \eqref{main free boundary} in $U$. Since $U$ is arbitrary, we get
$u_*$ is a viscosity supersolution of \eqref{main free boundary} in $\Omega$. Hence, $u$ is 
a discontinuous viscosity solution of \eqref{main free boundary}.
\end{proof}
\begin{Rem}
	By the same arguments as in the proof of Theorem \ref{the maximal},
	 if we assume further that $\Gamma, f$ and $\psi$ satisfy  the hypothesis of \cite[Theorem 1.1]{Li}
	then each maximal viscosity subsolution
	for \eqref{main free boundary} is approximated on every relatively compact subset of $\Omega$ by
	a decreasing sequence $u_j$ of classical solutions of \eqref{main free boundary}.
\end{Rem}
\begin{Rem}
	In general, if $u\in USC\cap L^{\infty}(\Omega)$
	 is a discontinuous viscosity solution of \eqref{main free boundary} then $u$ may not be a maximal visocity
	 subsolution. For example, let $\{a_j\}_{j=1}^{\infty}$ be a dense subset of the unit ball $\B^{2n}$ in
	 $\C^n$ and let
	 \begin{center}
	 	$v(z)=|z|^2+\sum\limits_{j=1}^{\infty}\dfrac{\log|z-a_j|}{2^j}.$
	 \end{center}
 We have $u=e^v$ is a bounded plurisubharmonic function and $u_*=0$ in $\B^{2n}$. Therefore, $u$
  is a discontinuous viscosity solution of the equation $F_n(Hw)=0$ in $\B^{2n}$. However, $u$ is
  not a maximal plurisubharmonic function in $\B^{2n}$, since its Monge-Amp\`ere measure 
  $(dd^cu)^n\geq e^{nv}(dd^cv)^n\geq e^{nv}(dd^c|z|^2)^n$ is not 
  identically $0$. Then $u$ is not a maximal viscosity solution of the equation $F_n(Hw)=0$ in $\B^{2n}$.
\end{Rem}

\end{document}